\documentclass[runningheads]{llncs}
\usepackage{graphicx}
\usepackage{longtable}
\usepackage{wrapfig}
\usepackage{ucs}
\usepackage{dsfont}
\usepackage{mathrsfs}
\usepackage{mathtools}
\usepackage{amsmath}
\usepackage{amsfonts}
\usepackage{amssymb}
\usepackage{soul}
\usepackage[makeroom]{cancel}
\usepackage{bbm}
\usepackage[english]{babel}
\usepackage{ucs}
\usepackage[colorlinks=true,linkcolor=blue,citecolor=blue,urlcolor=orange]{hyperref}
\usepackage{longtable}
\usepackage{wasysym}
\usepackage{verbatim}
\usepackage{multirow}
\usepackage{bussproofs}
\usepackage{setspace}
\usepackage{graphicx}
\usepackage{stmaryrd}
\usepackage{forest}
\forestset{smullyan tableaux/.style={for tree={math content},where n children=1{!1.before computing xy={l=\baselineskip},!1.no edge}{},closed/.style={label=below:$\times$},},}
\usepackage{longtable}
\usepackage[all]{xy}
\usepackage{wrapfig}
\usepackage{xcolor}
\usepackage[colorinlistoftodos,prependcaption,textsize=small]{todonotes}
\usepackage[most]{tcolorbox}
\usepackage{enumitem}
\usepackage{tikz}

\newcommand{\commentSabine}[1]{}

\newcommand{\Prop}{\mathtt{Prop}}
\newcommand{\coimplies}{\Yleft}
\newcommand{\BD}{\mathsf{BD}}
\newcommand{\pspace}{\mathsf{PSPACE}}
\newcommand{\biG}{\mathsf{biG}}
\newcommand{\Gsquare}{\mathsf{G}^2}
\newcommand{\KGsquare}{\mathbf{K}\mathsf{G}^2}
\newcommand{\fbbirelKGsquare}{\mathbf{K}\mathsf{G}^{2\pm}_{\mathsf{fb}}}
\newcommand{\KG}{\mathfrak{GK}}
\newcommand{\KbiG}{\mathbf{K}\mathsf{biG}}
\newcommand{\fbKGsquare}{\mathbf{K}\mathsf{G}^2_{\mathsf{fb}}}

\newcommand{\birelKGsquare}{\mathbf{K}\mathsf{G}^{2\pm}}
\newcommand{\bimodalLsquare}{\mathsf{bi}\mathcal{L}^\neg_{\Box,\lozenge}}

\newcommand{\Luk}{{\mathchoice{\mbox{\rm\L}}{\mbox{\rm\L}}{\mbox{\rm\scriptsize\L}}{\mbox{\rm\tiny\L}}}}

\newcommand{\Str}{\mathsf{Str}}
\newcommand{\AStr}{\mathsf{AStr}}
\newcommand{\str}{\mathsf{str}}
\newcommand{\real}{\mathsf{rl}}

\usepackage{comment}

\begin{document}

\setlength{\jot}{0pt} 
\setlength{\abovedisplayskip}{2pt}
\setlength{\belowdisplayskip}{2pt}
\setlength{\abovedisplayshortskip}{1pt}
\setlength{\belowdisplayshortskip}{1pt}

\title{Paraconsistent G\"{o}del modal logic\\on bi-relational frames\thanks{The research of Marta B\'ilkov\'a was supported by the grant 22-01137S of the Czech Science Foundation. The research of Sabine Frittella and Daniil Kozhemiachenko was funded by the grant ANR JCJC 2019, project PRELAP (ANR-19-CE48-0006). This research is part of the MOSAIC project financed by the European Union's Marie Sk\l{}odowska-Curie grant No.~101007627.}}
\titlerunning{Paraconsistent bi-relational G\"{o}del logic}
\author{Marta B\'ilkov\'a\inst{1}\orcidID{0000-0002-3490-2083} \and Sabine Frittella\inst{2}\orcidID{0000-0003-4736-8614}\and Daniil Kozhemiachenko\inst{2}\orcidID{0000-0002-1533-8034}}
\authorrunning{B\'ilkov\'a et al.}
\institute{The Czech Academy of Sciences, Institute of Computer Science, Prague\\
\email{bilkova@cs.cas.cz}
\and
INSA Centre Val de Loire, Univ.\ Orl\'{e}ans, LIFO EA 4022, France\\
\email{sabine.frittella@insa-cvl.fr, daniil.kozhemiachenko@insa-cvl.fr}}
\maketitle              
\begin{abstract}
We further develop the paraconsistent G\"{o}del modal logic $\KGsquare$. In this paper, we consider its version $\birelKGsquare$ endowed with Kripke semantics on $[0,1]$-valued frames with two fuzzy relations $R^+$ and $R^-$ (degrees of trust in assertions and denials) and two valuations $v_1$ and $v_2$ (support of truth and support of falsity) linked with a De Morgan negation $\neg$.

We demonstrate that in contrast to $\KGsquare$, $\birelKGsquare$ \emph{does not} extend G\"{o}del modal logic $\KG$ and that $\Box$ and $\lozenge$ are not interdefinable. We also show that several important classes of frames are $\birelKGsquare$ definable (in particular, crisp, mono-relational, and finitely branching). For $\birelKGsquare$ over finitely branching frames, we create a sound and complete constraint tableaux calculus and a decision procedure based upon it. Using the decision procedure we show that $\birelKGsquare$ satisfiability and validity are in $\mathsf{PSPACE}$.
\end{abstract}
\section{Introduction\label{sec:introduction}}
Paraconsistent logics are characterised by the rejection of the \emph{explosion} of entailment, i.e., $p,\neg p\models q$ (or, alternatively, not considering $p\wedge\neg p\rightarrow q$ to be valid). The motivation is very simple: even if we have contradictory premises, it is counterintuitive to think that literally every statement should follow from them.

Likewise, in modal logics, when interpreting $\Box$ or $\lozenge$ as, e.g., beliefs or obligations, the validity of $\Box(p\wedge\neg p)\rightarrow\Box q$ and $\lozenge(p\wedge\neg p)\rightarrow\lozenge q$ do not correspond to our notions of beliefs and obligations. Indeed, a~person can believe in one contradiction but not in another one. I can have conflicting moral obligations but I am not obliged to do absolutely everything. Thus, paraconsistent modal logics are more intuitive than the classical ones when formalising such notions.

\vspace{.5em}

\textbf{Paraconsistent modal logics}
One of the traditional approaches to the semantical representation of paraconsistency is to consider Kripke frames not with one but two valuations~\cite{Priest2008FromIftoIs,Wansing2005,Goble2006,Sherkhonov2008,OdintsovWansing2010,OdintsovWansing2017,Drobyshevich2020}. These are interpreted as independent supports of truth and falsity (or positive and negative supports). The idea follows Belnap's and Dunn's ‘useful four-valued logic’~\cite{Belnap2019} (alias, $\BD$ or $\mathsf{FDE}$ --- ‘first-degree entailment’).

An expected next step after introducing separate valuations of formulas' truth and falsity is to introduce separate accessibility relations as it is done in~\cite{Sherkhonov2008,Drobyshevich2020}: one relation is used to determine whether a modal formula \emph{is true at $w$}, and the other whether it is \emph{false at $w$}.

\vspace{.5em}

\textbf{Fuzzy modal logics}
In the above-mentioned cases, a formula can have only one of the four ‘Belnapian values’ at every given state: it is either \emph{true only}, \emph{false only}, \emph{both true and false}, or \emph{neither true nor false}. It is clear, however, that some statements have multiple \emph{truth degrees} ranging from $0$ to $1$: e.g., if $s$ is ‘the apple is sour’, I can sometimes \emph{partially agree} with $s$. This is the idea behind fuzzy logics.

This idea also makes sense when considering modal statements as well: I can have obligations of different strengths or I can believe in one proposition more than in the other. Thus, we have deontic, doxastic, epistemic fuzzy logics (cf., e.g.,~\cite{DellundeGodo2008} and~\cite{DankovaBehounek2020}).

Different (propositional) fuzzy logics have different expressive capacities. One can roughly divide them into three classes: the ones that can express (truncated) addition and subtraction; the ones that can express order on $[0,1]$; those that can do neither of those. The most well-known examples are, respectively, \L{}ukasiewicz, G\"{o}del, and Product logic (cf., e.g.,~\cite{Hajek1998} or~\cite{MetcalfeOlivettiGabbay2008} for further details).

When formalising natural-language modal statements, it is reasonable to expect that an agent can compare them (e.g., ‘I think that the rain today is more likely than a tornado’). On the other hand, it is rare to see somebody who says ‘I am 67\% certain that Paula's dog is a golden retriever’ while ‘I think that Paula's dog is rather a golden retriever than a dachshund’ is a completely natural sentence. Thus, the logics of the second kind seem to be the most reasonable choice.

\vspace{.5em}

\textbf{G\"{o}del modal logics and their paraconsistent expansions} (Propositional) G\"{o}del logic $\mathsf{G}$ can be thought of as a logic of comparative truth since the value of a formula depends not on the values but rather on the order of the variables. Thus, it is well-suited to the formalisation of modal statements. The expansion of $\mathsf{G}$ with $\Box$ and $\lozenge$ ($\KG$ in the notation of~\cite{RodriguezVidal2021}) with semantics on $[0,1]$-valued frames with fuzzy accessibility relations was first introduced in~\cite{CaicedoRodriguez2010} and since then well studied. In particular, axiomatisations of both $\Box$ and $\lozenge$ fragments and the axiomatisations of bi-modal fuzzy~\cite{CaicedoRodriguez2015} and crisp~\cite{RodriguezVidal2021} logics are known. Moreover, both fuzzy and crisp $\KG$, and its monomodal fragments are $\pspace$-complete~\cite{MetcalfeOlivetti2009,MetcalfeOlivetti2011,CaicedoMetcalfeRodriguezRogger2013,CaicedoMetcalfeRodriguezRogger2017}.

Bi-G\"{o}del (symmetric G\"{o}del in~\cite{GrigoliaKiseliovaOdisharia2016}) logic $\biG$ expands $\mathsf{F}$ with $\coimplies$ or Baaz' Delta operator $\triangle$. This allows expressing \emph{strict} order. Thus, modal expansions of $\biG$ can formalise statements such as ‘I think that Paula's dog is rather a~golden retriever than a dachshund’ given above where ‘rather’ is construed as ‘strictly more confident’. $\KbiG$ (the expansion of $\biG$ with $\Box$ and $\lozenge$) was introduced in~\cite{BilkovaFrittellaKozhemiachenko2022IJCAR} and given an axiomatisation in~\cite{BilkovaFrittellaKozhemiachenko2022IGPLarxiv}. Additionally, a temporal expansion of bi-G\"{o}del logic was introduced in~\cite{AguileraDieguezFernandez-DuqueMcLean2022}. The satisfiability and validity of both logics are also in $\pspace$.

Finally, just as intuitionistic logic can be expanded with a De Morgan negation, so can be G\"{o}del logic as well (cf.~\cite{Wansing2008} for expansions of intuitionistic propositional logic and~\cite{BilkovaFrittellaKozhemiachenko2021,BilkovaFrittellaKozhemiachenkoMajer2023IJAR} for expansions of $\mathsf{G}$). The negated formula $\neg\phi$ is, as expected, interpreted as ‘$\phi$ is false’. In this paper, we are focussing on the expansion of $\mathsf{G}$ that defines $\neg(\phi\rightarrow\phi')\leftrightarrow(\neg\phi'\coimplies\neg\phi)$ and $\neg(\phi\coimplies\phi')\leftrightarrow(\neg\phi'\rightarrow\neg\phi)$ after $\mathsf{I}_4\mathsf{C}_4$\footnote{The logic was introduced independently by different authors~\cite{Wansing2008,Leitgeb2019}, and further studied in~\cite{OdintsovWansing2021}. It is the propositional fragment of Moisil's modal logic~\cite{Moisil1942}. We are grateful to Heinrich Wansing who pointed this out to us.} from~\cite{Wansing2008}. We henceforth call this logic $\Gsquare$.

$\Gsquare$ has some nice properties. First, all $\Gsquare$ connectives have their duals. Second, in contrast to $\mathsf{G}$, \emph{it is not the case that either $p\rightarrow q$ or $q\rightarrow p$ has designated value under any valuation}. This means that not all statements are comparable. Indeed, when reasoning about beliefs, it is safe to assume that they are not always comparable: people do not have to believe that a thunderstorm is going to happen today \emph{more (or less) than} they believe that their distant relative has two dogs.

The crisp modal expansion of $\Gsquare$ with $\Box$ and $\lozenge$\footnote{Note that in the presence of $\neg$, we have $\neg\Box\neg\phi\leftrightarrow\lozenge\phi$ and $\neg\lozenge\neg\phi\leftrightarrow\Box\phi$.} was introduced in~\cite{BilkovaFrittellaKozhemiachenko2022IJCAR} and axiomatised in~\cite{BilkovaFrittellaKozhemiachenko2022IGPLarxiv}. In this paper, we make the expected next step: we permit fuzzy frames in $\KGsquare$ and also equip them with \emph{two} relations: $R^+$ and $R^-$ with which we determine, respectively, the positive and negative supports of modal formulas. We call this logic $\birelKGsquare$.

\vspace{.5em}

\textbf{Plan of the paper}
The remaining part of the text is organised as follows. Section~\ref{sec:language} contains the necessary logical preliminaries as well as the semantical definitions for $\birelKGsquare$. In Section~\ref{sec:comparison}, we establish relations between $\KbiG$, $\KGsquare$, and $\birelKGsquare$. Then, in Section~\ref{sec:semantics}, we investigate the correspondence between classes on frames and formulas valid on them. In Section~\ref{sec:tableaux}, we design a sound and complete tableaux system for the $\birelKGsquare$ over finitely branching frames and prove that its satisfiability and validity are $\pspace$-complete. Finally, in Section~\ref{sec:conclusion}, we summarise the results of the paper and sketch the plan for future work.
\section{Language and semantics\label{sec:language}}
To make the exposition self-contained, we begin with the algebraic semantics of $\biG$ and then use it to obtain Kripke semantics of $\birelKGsquare$.
\begin{definition}\label{def:bi-G_algebra}
The bi-G\"{o}del algebra $[0,1]_{\mathsf{G}}=\langle[0,1],0,1,\wedge_\mathsf{G},\vee_\mathsf{G},\rightarrow_{\mathsf{G}},\coimplies,\triangle,{\sim}\rangle$ is defined as follows: for all $a,b\in[0,1]$, we have $a\wedge_\mathsf{G}b=\min(a,b)$, $a\vee_\mathsf{G}b=\max(a,b)$. The remaining operations are defined below:
\begin{align*}
a\rightarrow_\mathsf{G}b&=
\begin{cases}
1,\text{ if }a\leq b\\
b\text{ else}
\end{cases}
&
a\coimplies_\mathsf{G}b&=
\begin{cases}
0,\text{ if }a\leq b\\
a\text{ else}
\end{cases}
\\
{\sim_\mathsf{G}}a&=
\begin{cases}
0,\text{ if }a>0\\
1\text{ else}
\end{cases}
&
\triangle_\mathsf{G}a&=
\begin{cases}
0,\text{ if }a<1\\
1\text{ else}
\end{cases}
\end{align*}
\end{definition}
Note that $\triangle_\mathsf{G}$ and ${\coimplies_\mathsf{G}}$ are interdefinable:
\begin{align*}
\triangle_\mathsf{G}a&=1\coimplies_\mathsf{G}(1\coimplies_\mathsf{G}a)&a\coimplies_\mathsf{G}b&=a\wedge{\sim_\mathsf{G}}\triangle_\mathsf{G}(a\rightarrow b)
\end{align*}
\begin{definition}\label{def:frames}
A \emph{bi-relational frame} is a tuple $\mathfrak{F}\!=\!\langle W,R^+,R^-\rangle$ with $W\!\neq\!\varnothing$ and
\begin{enumerate}[noitemsep,topsep=2pt]
\item $R^+,R^-:W\times W\rightarrow\{0,1\}$ if $\mathfrak{F}$ is \emph{crisp}\footnote{Equivalently, $R^+$ and $R^-$ in crisp frames are \emph{relations} on $W$.};
\item $R^+,R^-:W\times W\rightarrow[0,1]$ if $\mathfrak{F}$ is \emph{fuzzy}.
\end{enumerate}
\end{definition}
\begin{definition}[Language and semantics]\label{def:semantics}
We fix a countable set $\Prop$ of propositional variables and define the language $\bimodalLsquare$ via the following grammar in BNF.
\[\bimodalLsquare\ni\phi\coloneqq p\in\Prop\mid\neg\phi\mid(\phi\wedge\phi)\mid(\phi\rightarrow\phi)\mid\Box\phi\mid\lozenge\phi\]
Constants $\mathbf{0}$ and $\mathbf{1}$, $\vee$, $\coimplies$, $\triangle$, and G\"{o}del negation ${\sim}$ can be defined as usual:
\begin{align*}
\mathbf{1}&\coloneqq p\!\rightarrow\!p&\mathbf{0}&\coloneqq\neg\mathbf{1}&{\sim}\phi&\coloneqq\phi\!\rightarrow\!\mathbf{0}\\\triangle\phi&\coloneqq\mathbf{1}\coimplies(\mathbf{1}\coimplies\phi)&\phi\!\vee\!\phi'&\coloneqq\neg(\neg\phi\!\wedge\!\neg\phi')&\phi\!\coimplies\!\phi'&\coloneqq\neg(\neg\phi'\!\!\rightarrow\!\!\neg\phi)
\end{align*}

A \emph{$\birelKGsquare$ model} is a tuple $\mathfrak{M}=\langle W,R^+,R^-,v_1,v_2\rangle$ with $\langle W,R^+,R^-\rangle$ being a~crisp or fuzzy frame and $v_1,v_2:\Prop\rightarrow[0,1]$ that are extended to complex formulas as follows.
\begin{center}
\begin{tabular}{rclrcl}
$v_1(\neg\phi,w)$&$=$&$v_2(\phi,w)$&$v_2(\neg\phi,w)$&$=$&$v_1(\phi,w)$\\
$v_1(\phi\wedge\phi',w)$&$=$&$v_1(\phi,w)\wedge_\mathsf{G}v_1(\phi',w)$&$v_2(\phi\wedge\phi',w)$&$=$&$v_2(\phi,w)\vee_\mathsf{G}v_2(\phi',w)$\\
$v_1(\phi\rightarrow\phi',w)$&$=$&$v_1(\phi,w)\!\rightarrow_\mathsf{G}\!v_1(\phi',w)$&$v_2(\phi\rightarrow\phi',w)$&$=$&$v_2(\phi',w)\coimplies_\mathsf{G}v_2(\phi,w)$
\end{tabular}
\end{center}

The semantics of modalities is as follows.
\begin{center}
\begin{tabular}{rclrcl}
$v_1(\Box\phi,w)$&$=$&$\inf\limits_{w'\!\in\!W}\!\{wR^+w'\!\!\rightarrow_\mathsf{G}\!\!v_1(\phi,w')\}$
&
$v_2(\Box\phi,w)$&$=$&$\sup\limits_{w'\!\in\!W}\!\{wR^-w'\!\!\wedge_\mathsf{G}\!\!v_2(\phi,w')\}$\\
$v_1(\lozenge\phi,w)$&$=$&$\sup\limits_{w'\!\in\!W}\!\{wR^+w'\!\wedge_\mathsf{G}\!v_1(\phi,w')\}$
&
$v_2(\lozenge\phi,w)$&$=$&$\sup\limits_{w'\!\in\!W}\!\{wR^-w'\!\rightarrow_\mathsf{G}\!v_2(\phi,w')\}$
\end{tabular}
\end{center}

We will further write $v(\phi,w)=(x,y)$ to designate that $v_1(\phi,w)=x$ and $v_2(\phi,w)=y$. We also set $S(w)=\{w':wSw'>0\}$.

We say that $\phi$ is \emph{valid on $\mathfrak{F}$} ($\mathfrak{F}\models\phi$) iff for every $v_1$ and $v_2$ on $\mathfrak{F}$ and every $w\!\in\!\mathfrak{F}$, it holds that $v(\phi,w)\!=\!(1,0)$. $\phi$ is \emph{$\birelKGsquare$ valid} iff it is valid on every frame.
\end{definition}

\begin{remark}\label{rem:v1equivalence}
Observe that the semantic conditions of the support of the truth ($v_1$) coincide with the semantics of $\KbiG$. Furthermore, the semantics of $\KGsquare$ can be retrieved by assuming that $R^+$ and $R^-$ are crisp and $R^+=R^-$. Cf.~\cite{BilkovaFrittellaKozhemiachenko2022IJCAR} for more detailed semantics of the logics.
\end{remark}
\section{$\KbiG$, $\KGsquare$, and $\birelKGsquare$\label{sec:comparison}}
Definition~\ref{def:semantics} gives a~reason to believe that $\birelKGsquare$ is in a sense intermediate between crisp $\KbiG$ and $\KGsquare$. In this section, we investigate the following questions.
\begin{enumerate}[noitemsep,topsep=2pt]
\item $\Box$ and $\lozenge$ are not interdefinable in $\KbiG$~\cite[Proposition~3]{BilkovaFrittellaKozhemiachenko2022IJCAR} but $\neg\Box\neg p\leftrightarrow\lozenge p$ and $\lozenge p\leftrightarrow\neg\Box\neg p$ are $\KGsquare$ valid. Are $\Box$ and $\lozenge$ interdefinable in $\birelKGsquare$?
\item $\KGsquare$ extends crisp $\KbiG$ and is conservative w.r.t.\ $\neg$-free formulas~\cite[Proposition~2]{BilkovaFrittellaKozhemiachenko2022IJCAR}. Does $\birelKGsquare$ (on mono- or bi-relational frames) extend fuzzy $\KbiG$? Does crisp $\birelKGsquare$ on bi-relational frames extend crisp $\KbiG$?
\end{enumerate}

We first show that $\Box$ and $\lozenge$ are, in fact, not interdefinable in $\birelKGsquare$.
\begin{theorem}\label{theorem:nondefinability}
$\Box$ and $\lozenge$ are not interdefinable.
\end{theorem}
\begin{proof}
Denote with $\mathcal{L}_\Box$ and $\mathcal{L}_\lozenge$ the $\lozenge$-free and $\Box$-free fragments of $\bimodalLsquare$, respectively. To prove the statement, it suffices to find a pointed model $\langle\mathfrak{M},w\rangle$ s.t.\ there is no $\mathcal{L}_\lozenge$ formula has the same value at $w$ as $\Box p$ and vice versa.

Consider the following model (all variables have the same values in all states exemplified by $p$). We have $v(\Box p,w_0)=\left(\frac{3}{5},\frac{3}{4}\right)$ and $v(\lozenge p,w_0)=\left(\frac{4}{5},\frac{2}{4}\right)$.
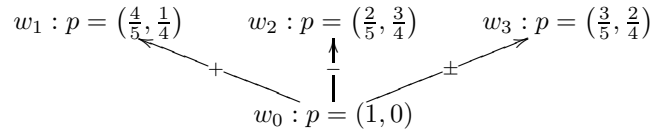
\[\xymatrix{w_1:p=\left(\frac{4}{5},\frac{1}{4}\right)&w_2:p=\left(\frac{2}{5},\frac{3}{4}\right)&w_3:p=\left(\frac{3}{5},\frac{2}{4}\right)\\&w_0:p=(1,0)\ar[u]|{-}\ar[ur]|{\pm}\ar[ul]|{+}&}\]

It is easy to check that $v(\phi,t)\in\{v(p,t),v(\neg p,t),(1,0),(0,1)\}$ for every $\phi\!\in\!\bimodalLsquare$ over one variable on the single-point irreflexive frame with a state $t$. Thus, for every $\chi\in\mathcal{L}_\Box$ and every $\psi\!\in\!\mathcal{L}_\lozenge$ it holds that
\begin{align*}
v(\Box\chi,w_0)&\in\left\{(0;1),\left(\frac{3}{5};\frac{3}{4}\right),\left(\frac{1}{4};\frac{3}{5}\right),\left(\frac{3}{4};\frac{3}{5}\right),\left(\frac{3}{5};\frac{1}{4}\right),(1;0)\right\}=X\\
v(\lozenge\psi,w_0)&\in\left\{(0;1),\left(\frac{4}{5};\frac{2}{4}\right),\left(\frac{2}{4};\frac{2}{5}\right),\left(\frac{2}{4};\frac{4}{5}\right),\left(\frac{2}{5};\frac{2}{4}\right),(1;0)\right\}=Y
\end{align*}
Now, let $X^c$ and $Y^c$ be the closures of $X$ and $Y$ under propositional operations. It is clear\footnote{Cf.~the full sets in the appendix.} that $\left(\frac{3}{5};\frac{3}{4}\right)\notin Y^c$ and $\left(\frac{4}{5};\frac{2}{4}\right)\notin X^c$. It is also easy to verify by induction that for all $\chi'\in\mathcal{L}_\Box$ and $\psi'\in\mathcal{L}_\lozenge$, it holds that $v(\chi',w_0)\in X^c$ and $v(\psi',w_0)\in Y^c$. The result now follows.
\end{proof}

The next statement gives the negative answer to the first half of the second question.
\begin{theorem}\label{theorem:noextension}
Fuzzy $\birelKGsquare$ does not extend fuzzy $\KbiG$.
\end{theorem}
\begin{proof}
Recall that $\lozenge{\sim\sim}p\rightarrow{\sim\sim}\lozenge p$ is a theorem of fuzzy G\"{o}del modal logic and that $\KbiG$ extends fuzzy $\KG$~\cite[Proposition~2]{BilkovaFrittellaKozhemiachenko2022IJCAR}. Thus, $\lozenge{\sim\sim}p\rightarrow{\sim\sim}\lozenge p$ is $\KbiG$ valid. Consider the model below.
\[\xymatrix{w~\ar[rr]^(.3){R^+=R^-=\frac{1}{2}}&&~w':p=\left(1,\frac{2}{3}\right)}\]
It is clear that $v_2({\sim\sim}\lozenge p,w)=1$ but $v_2(\lozenge{\sim\sim}p,w)=0$. Thus, $v_2(\lozenge{\sim\sim}p\rightarrow{\sim\sim}\lozenge p,w)=1$, i.e., $\lozenge{\sim\sim}p\rightarrow{\sim\sim}\lozenge p$ is not valid in $\birelKGsquare$.
\end{proof}

Note that we used a \emph{mono-relational} fuzzy frame in the proof of the above theorem. It remains to consider the \emph{crisp $\birelKGsquare$ over bi-relational frames}. In the remainder of the section, we show that it \emph{does extend} crisp $\KbiG$. The next lemma is a straightforward adaptation of~\cite[Proposition~1]{BilkovaFrittellaKozhemiachenko2022IJCAR}.
\begin{lemma}\label{lemma:splitconflation}
Let $\mathfrak{M}=\langle W,R^+,R^-,v_1,v_2\rangle$ be a crisp model
. Let further $\mathfrak{M}^*=\langle W,(R^+)^*,(R^-)^*,v^*_1,v^*_2\rangle$ be as follows: $(R^+)^*=R^-$, $(R^-)^*=R^+$, $v^*_1(p,w)=1-v_2(p,w)$, and $v^*_2(p,w)=1-v_1(p,w)$.

Then, $v(\phi,w)=(x,y)$ iff $v^*(\phi,w)=(1-y,1-x)$.
\end{lemma}
\begin{proof}
We proceed by induction on $\phi$. The basis case of propositional variables holds by the construction of $\mathfrak{M}^*$. The cases of propositional connectives hold by~\cite[Proposition~5]{BilkovaFrittellaKozhemiachenko2021}. We consider the case of $\phi=\Box\psi$.

Let $v(\Box\psi,w)=(x,y)$. Then $\inf\{v_1(\psi,w'):wR^+w'\}=x$, and $\sup\{v_2(\psi,w'):wR^-w'\}=y$. Now, we apply the induction hypothesis to $\psi$, and thus if $v(\psi,s)=(x',y')$, then $v^*_1(\psi,s)=1-y'$ and $v^*_2(\psi,s')=1-x'$ for any $s\in R^+(w)=(R^-)^*(w)$ and $s'\in R^-(w)=(R^+)^*(w)$. But then $\inf\{v^*_1(\psi,w'):w(R^+)^*w'\}=1-y$, and $\sup\{v^*_2(\psi,w'):w(R^-)^*w'\}=1-x$, as required.
\end{proof}
\begin{theorem}\label{theorem:crispextension}
Let $\phi$ be a $\neg$-free formula. Then, $\phi$ is crisp $\KbiG$ valid iff it is crisp $\birelKGsquare$ valid.
\end{theorem}
\begin{proof}
It is clear that if $\phi$ is \emph{not} $\KbiG$ valid, then it is not $\birelKGsquare$ valid either (recall Remark~\ref{rem:v1equivalence}). For the converse, it follows from Lemma~\ref{lemma:splitconflation}, that if $v_2(\phi,w)>0$ for some frame $\mathfrak{F}=\langle W,R^+,R^-\rangle$, $w\in\mathfrak{F}$ and $v_2$ on $\mathfrak{F}$, then $v^*_1(\phi,w)<1$. But $\phi$ does not contain $\neg$ and thus depends only on $v_2$ and $R^-$, whence $v^*_1$ is a $\KbiG$ valuation on $\langle W,R^-\rangle$. Thus, $\phi$ is not $\KbiG$ valid either.
\end{proof}

Note that Lemma~\ref{lemma:splitconflation} implies that in order to check crisp $\birelKGsquare$ validity of $\phi$, it suffices to check whether it is always the case that $v_1(\phi,w)=1$ in split models. On the other hand, this reduction, evidently, does not hold for fuzzy $\birelKGsquare$ due to Theorem~\ref{theorem:noextension}.
\section{Correspondence theory and frame definability\label{sec:semantics}}
In this section, we investigate the modal (un)definability of frame classes in crisp and fuzzy $\birelKGsquare$. We begin with corollaries of Lemma~\ref{lemma:splitconflation} that concern the definability of crisp frames.
\begin{definition}\label{def:framecounterparts}
Let $\mathbb{K}$ be a class of crisp frames. A first- or second-order formula $F$ \emph{defines} $\mathbb{K}$ iff for every $\mathfrak{F}=\langle W,R\rangle$, it holds that $\mathfrak{F}\in\mathbb{K}$ iff $\mathfrak{F}\models F(R)$.
\begin{enumerate}[noitemsep,topsep=2pt]
\item The \emph{$+$-counterpart} of $\mathbb{K}$ is the class $\mathbb{K}^+$ of frames $\mathfrak{F}=\langle W,R^+,R^-\rangle$ s.t.\ $\langle W,R^+\rangle\models F(R^+)$.
\item The \emph{$-$-counterpart} of $\mathbb{K}$ is the class $\mathbb{K}^-$ of frames $\mathfrak{F}=\langle W,R^+,R^-\rangle$ s.t.\ $\langle W,R^-\rangle\models F(R^-)$.
\item The \emph{$\pm$-counterpart} of $\mathbb{K}$ is the class $\mathbb{K}^\pm=\mathbb{K}^+\cap\mathbb{K}^-$.
\end{enumerate}
\end{definition}
\begin{corollary}\label{cor:definabilitypreservation}
Let $\phi$ be a $\neg$-free formula that defines a class of frames $\mathfrak{F}=\langle W,R\rangle$ $\mathbb{K}$ in $\KbiG$ and let $\mathbb{K}^\pm$ be the $\pm$-counterpart of $\mathbb{K}$. Then $\phi$ defines $\mathbb{K}^\pm$ in $\birelKGsquare$.
\end{corollary}
\begin{proof}
Assume that $\phi$ \emph{does not} define $\mathbb{K}^\pm$. Then, either (1) there is $\mathfrak{F}\notin\mathbb{K}^\pm$ s.t.\ $\mathfrak{F}\models\phi$ or (2) $\mathfrak{H}\not\models\phi$ for some $\mathfrak{H}\in\mathbb{K}^\pm$. Since $\phi$ defines $\mathbb{K}$ in $\KbiG$, it is clear that $\mathfrak{F},\mathfrak{H}\in\mathbb{K}^+$. Thus, we need to reason for contradiction in the case when $\mathfrak{F}\notin\mathbb{K}^-$ or $\mathfrak{H}\notin\mathbb{K}^-$. We prove only (1) as (2) can be tackled in a dual manner.

Observe that $v_2(\phi,w)\!=\!0$ for every $w\!\in\!\mathfrak{F}$ and $v_2$ defined on $\mathfrak{F}\!=\!\langle W,R^+,R^-\rangle$. But then, by Lemma~\ref{lemma:splitconflation} we have that $v^*_1(\phi,w)=1$ for every $w\in\mathfrak{F}$ and $v^*_1$ defined on $\mathfrak{F}$. Thus, since for every $v^*_1$ there is $v_2$ from which it could be obtained, $\phi$ is $\KbiG$ valid on a frame $\langle W,R^-\rangle$ where $R^-$ \emph{is not definable} via $F(R^-)$. Hence, $\phi$~does not define $\mathbb{K}$ in $\KbiG$ either. A contradiction.
\end{proof}

A natural question now is whether it is possible to have definitions of classes of frames that are only $+$-counterparts (or $-$-counterparts) of $\KbiG$-definable frame classes. E.g., a class of frames whose $R^+$ is reflexive but $R^-$ is not necessarily so. The next statement provides a negative answer.
\begin{corollary}\label{cor:partialnondefinability}
Let $F(R^+)$ and $F(R^-)$ be two first- or second-order formulas defining relations $R^+$ and $R^-$. Then, the class $\mathbb{K}$ of crisp frames $\mathfrak{F}\!=\!\langle W,R^+,R^-\rangle$ with only $R^+$ being definable by $F(R^+)$ (resp., only $R^-$ being definable by $F(R^-)$) \emph{is not definable in $\birelKGsquare$}.
\end{corollary}
\begin{proof}
We reason for contradiction. Assume that $\phi$ defines $\mathbb{K}$, and let $\mathfrak{F}\in\mathbb{K}$ with $\mathfrak{F}=\langle W,R^+,R^-\rangle$ s.t.\ $F(R^-)$ does not hold of $\mathfrak{F}$. Now denote $\mathfrak{F}^*=\langle W,R^-,R^+\rangle$. Clearly, $\mathfrak{F}^*\notin\mathbb{K}$. However, by Lemma~\ref{lemma:splitconflation}, we have that $\mathfrak{F}^*\models\phi$, i.e., $\phi$ does not define $\mathbb{K}$. A contradiction.
\end{proof}

As of now, we have discussed the definability of \emph{different classes of crisp frames}. $\Box(p\vee q)\rightarrow(\Box p\vee\lozenge q)$ defines crisp frames in $\KG$~\cite{RodriguezVidal2021} and $\triangle\Box p\rightarrow\Box\triangle p$ in $\KbiG$~\cite{BilkovaFrittellaKozhemiachenko2022IGPLarxiv}\footnote{It makes sense to speak of definability of crisp frames in $\KG$ and $\KbiG$ separately since $\triangle\Box p\rightarrow\Box\triangle p$ is essential in the completeness proof.}, however, $\birelKGsquare$ does not extend $\KbiG$ (nor $\KG$), whence the definability of the \emph{class of all crisp frames} is not immediate.

\begin{theorem}\label{theorem:crispdefinition}
Let $\mathfrak{F}=\langle W,R^+,R^-\rangle$.
\begin{enumerate}[noitemsep,topsep=2pt]
\item $R^+$ is \emph{crisp} iff $\mathfrak{F}\models\triangle\Box p\rightarrow\Box\triangle p$.
\item $R^-$ is \emph{crisp} iff $\mathfrak{F}\models\lozenge{\sim\sim}p\rightarrow{\sim\sim}\lozenge p$.
\end{enumerate}
\end{theorem}
\begin{proof}
Note, first of all, that $v_i(\triangle\phi,w),v_i({\sim\sim}\phi,w)\in\{0,1\}$ for every $\phi$ and $i\in\{1,2\}$. Now let $R^+$ be crisp. We have
\begin{align*}
v_1(\triangle\Box p,w)=1&\text{ then }v_1(\Box p,w)=1\\
&\text{ then }\inf\{v_1(p,w'):wR^+w'\}=1\tag{$R^+$ is crisp}\\
&\text{ then }\inf\{v_1(\triangle p,w'):wR^+w'\}=1\\
&\text{ then }v_1(\Box\triangle p,w)=1
\end{align*}
\begin{align*}
v_2(\triangle\Box p,w)=0&\text{ then }v_2(\Box p,w)=0\\
&\text{ then }\sup\limits_{w'\in W}\{wR^-w'\wedge_\mathsf{G}v_2(p,w')\}=0\\
&\text{ then }\sup\limits_{w'\in W}\{wR^-w'\wedge_\mathsf{G}v_2(\triangle p,w')\}=0\\
&\text{ then }v_2(\Box\triangle p,w)=0
\end{align*}
For the converse, let $wR^+w'=x$ with $x\in(0,1)$. We set $v(p,w')=(x,0)$ and $v(p,w'')=(1,0)$ elsewhere. It is clear that $v(\triangle\Box p,w)=(1,0)$ but $v(\Box\triangle p,w)=(0,0)$. Thus, $v(\triangle\Box p\rightarrow\Box\triangle p,w)\neq(1,0)$, as required.

The case of $R^-$ is considered dually. For crisp $R^-$, we have
\begin{align*}
v_1(\lozenge{\sim\sim}p,w)=1&\text{ then }\sup\limits_{w'\in W}\{wR^-w'\wedge_\mathsf{G}v_1({\sim\sim}p,w')\}=1\\
&\text{ then }\sup\limits_{w'\in W}\{wR^-w'\wedge_\mathsf{G}v_2(p,w')\}>0\\
&\text{ then }v_1(\lozenge p,w)>0\\
&\text{ then }v_1({\sim\sim}\lozenge p,w)=1
\end{align*}
\begin{align*}
v_2(\lozenge{\sim\sim}p,w)=0&\text{ then }\inf\{v_2({\sim\sim}p,w'):wR^-w'\}=0\tag{$R^-$ is crisp}\\
&\text{ then }\inf\{v_2(p,w'):wR^-w'\}<1\\
&\text{ then }v_2(\lozenge p,w)<1\\
&\text{ then }v_2({\sim\sim}\lozenge p,w)=0
\end{align*}
For the converse, let $wR^-w'=y$ with $y\in(0,1)$. We set $v(p,w')=(1,y)$ and $v(p,w'')=(1,0)$ elsewhere. It is clear that $v(\lozenge{\sim\sim} p,w)=(1,0)$ but $v({\sim\sim}\lozenge p,w)=(1,1)$. Thus, $v(\lozenge{\sim\sim}p\rightarrow{\sim\sim}\lozenge p,w)\neq(1,0)$, as required.
\end{proof}

The above statement highlights an important contrast between crisp and fuzzy bi-relational frames: while it is impossible to define $R^+$ and $R^-$ separately in crisp frames, we can define a class of frames where only $R^+$ (or only $R^-$) is crisp. It is now instructive to ask whether we can define some relations \emph{between} $R^+$ and $R^-$. In particular, we show that
\begin{enumerate}[noitemsep,topsep=2pt]
\item frames where there are $w$ and $w'$ s.t.\ $wR^+w',wR^-w'>0$ are not definable;
\item mono-relational frames (both crisp and fuzzy) are definable.
\end{enumerate}

\begin{definition}\label{def:splitting}
Let $\mathfrak{M}=\langle W,R^+,R^-,v_1,v_2\rangle$ be a~model. We define its \emph{splitting} to be $\mathfrak{M}^\mathsf{s}=\langle W^\mathsf{s},(R^+)^\mathsf{s},(R^-)^\mathsf{s},v_1^\mathsf{s},v_2^\mathsf{s}\rangle$ with
\begin{itemize}[noitemsep,topsep=2pt]
\item $W^\mathsf{s}=\{\ulcorner wSw'\urcorner,\ulcorner\varnothing Su\urcorner:wSw',S\in\{R^+,R^-\},\neg\exists t:tSu>0\}$;
\item $\ulcorner uSu'\urcorner S^\mathsf{s}\ulcorner wSw'\urcorner=u'Sw'$ with $S\in R^+,R^-$;
\item for every $\ulcorner wSw'\urcorner$ and $i\in\{1,2\}$, $v^\mathsf{s}_i(p,\ulcorner wSw'\urcorner)=v_i(p,w')$.
\end{itemize}

We will further denote
\begin{align*}
\llbracket w\rrbracket&=\{\ulcorner\varnothing Sw\urcorner,\ulcorner uSw\urcorner:S\in\{R^+,R^-\}\}
\end{align*}
\end{definition}
It is clear that there are no $u,u'\!\in\!W^\mathsf{s}$ s.t.\ $u(R^+)^\mathsf{s}u',u(R^-)^\mathsf{s}u'>0$. We will further call such models \emph{split models}.

The next statement is easy to prove.
\begin{lemma}\label{lemma:splitequivalence}
Let $\mathfrak{M}=\langle W,R^+,R^-,v_1,v_2\rangle$ be a~model and $\mathfrak{M}^\mathsf{s}$ be its splitting. Then $v_i(\phi,w)=v^\mathsf{s}_i(\phi,w^\mathsf{s})$ for every $\phi\in\bimodalLsquare$ and $w^\mathsf{s}\in\llbracket w\rrbracket$.
\end{lemma}
\begin{proof}
We proceed by induction. The basis case of propositional variables holds by the construction of $\mathfrak{M}^\mathsf{s}$. The cases of propositional connectives are straightforward. We consider the case of $\phi=\Box\chi$ (the $\lozenge$ case can be considered dually).

Let $\ulcorner uR^+w\urcorner\in\llbracket w\rrbracket$ be arbitrary. We have
\begin{align*}
v^\mathsf{s}_1(\Box\chi,\ulcorner uR^+w\urcorner)&=\inf\limits_{\ulcorner wR^+w'\urcorner\in W^\mathsf{s}}\{\ulcorner uR^+w\urcorner(R^+)^\mathsf{s}\ulcorner wR^+w'\urcorner\rightarrow_\mathsf{G}v^\mathsf{s}_1(\chi,\ulcorner wR^+w'\urcorner)\}\\
&=\inf\limits_{w'\in W}\{wR^+w'\rightarrow_{\mathsf{G}}v_1(\chi,w')\}\tag{by IH}\\
&=v_1(\Box\chi,w)
\end{align*}
\begin{align*}
v^\mathsf{s}_2(\Box\chi,\ulcorner uR^+w\urcorner)&=\sup\limits_{\ulcorner wR^-w'\urcorner\in W^\mathsf{s}}\{\ulcorner uR^-w\urcorner(R^-)^\mathsf{s}\ulcorner wR^-w'\urcorner\wedge_\mathsf{G}v^\mathsf{s}_2(\chi,\ulcorner wR^-w'\urcorner)\}\\
&=\sup\limits_{w'\in W}\{wR^-w'\wedge_\mathsf{G}v_2(\chi,w')\}\tag{by IH}\\
&=v_2(\Box\chi,w)
\end{align*}

Note that we could apply the induction hypothesis because $\ulcorner uR^+w\urcorner\in\llbracket w\rrbracket$, $\ulcorner wR^+w'\urcorner\!\in\!\llbracket w'\rrbracket$, and the values of $wR^+w'$ (resp., $wR^-w'$) become the values of $\ulcorner uR^+w\urcorner (R^+)^\mathsf{s}\ulcorner wR^+w'\urcorner$ (resp., $\ulcorner uR^-w\urcorner (R^-)^\mathsf{s}\ulcorner wR^-w'\urcorner$). The result follows.
\end{proof}

The following corollary is now immediate.
\begin{corollary}\label{cor:nointersection}
The class of (crisp or fuzzy) frames $\mathfrak{F}=\langle W,R^+,R^-\rangle$ s.t.\ both $wR^+w'>0$ and $wR^-w'>0$ for some $w,w'\in\mathfrak{F}$ is not definable in $\birelKGsquare$.
\end{corollary}

Let us now prove the definability of mono-relational frames.
\begin{theorem}\label{theorem:1relationdefinable}
Let $\mathfrak{F}=\langle W,R^+,R^-\rangle$. Then $\mathfrak{F}\models\Box p\leftrightarrow\neg\lozenge\neg p$ iff $R^+=R^-$.
\end{theorem}
\begin{proof}
Let $R^+=R^-$, we have
\begin{align*}
v_1(\neg\lozenge\neg p,w)&=v_2(\lozenge\neg p,w)\\
&=\inf\limits_{w'\in W}\{wRw'\rightarrow_\mathsf{G}v_2(\neg p,w')\}\\
&=\inf\limits_{w'\in W}\{wRw'\rightarrow_\mathsf{G}v_1(p,w')\}\\
&=v_1(\Box p,w)
\end{align*}
$v_2$ can be tackled similarly.

Now let $R^+\neq R^-$, i.e., $wR^+w'=x$ and $wR^-w'=y$ for some $w,w'\in\mathfrak{F}$, and assume w.l.o.g.\ that $x>y$. We define the values of $p$ as follows: $v(p,w'')=(1,0)$ for all $w''\neq w'$ and $v(p,w')=(x,y)$. It is clear that $v(\Box p,w)=(1,0)$ but $v(\neg\lozenge\neg p,w)=(y,x)\neq(1,0)$, as required.
\end{proof}

In the remaining part of the paper, we will be considering (fuzzy) $\birelKGsquare$ over \emph{finitely branching frames}, i.e., frames $\langle W,R^+,R^-\rangle$ where $|R^+(w)|,|R^-(w)|<\aleph_0$ for every $w\in W$. We will denote this logic $\fbbirelKGsquare$. This is a natural restriction for several reasons. First, if we use the logic to formalise reasoning, it makes sense to assume that agents can consider only a finite number of alternatives. This assumption is usually implicit in classical modal logics since they are oftentimes complete w.r.t.\ finitely branching (or even finite) models~\cite{FaginHalpernMosesVardi2003}. In modal expansions of G\"{o}del logic, this is not the case, and thus this assumption has to be forced. Second, infinitely branching frames can have \emph{unwitnessed} models: e.g., it is possible that $v(\lozenge p,w)=(1,0)$ but there are no $w'\in R^+(w)$ and $w''\in R^-(w)$ s.t.\ $v_1(p,w')=1$ and $v_2(p,w'')=0$. This makes the doxastic interpretation of modalities counterintuitive.

We finish the section by showing that fuzzy and crisp\footnote{In fact, the definability of \emph{crisp} finitely branching frames follows from Corollary~\ref{cor:definabilitypreservation} since ${\sim\sim}\Box(p\vee{\sim}p)$ defines finitely branching frames in $\KbiG$~\cite[Proposition~4.3]{BilkovaFrittellaKozhemiachenko2022IGPLarxiv}} finitely branching frames are definable. Note, however, that now we need \emph{two formulas}.
\begin{theorem}\label{theorem:finitebranching}
$\mathfrak{F}$ is finitely branching iff $\mathfrak{F}\models{\sim\sim}\Box(p\!\vee\!{\sim}p)$ and $\mathfrak{F}\models\mathbf{1}\!\coimplies\!\lozenge\neg(p\!\vee\!{\sim}p)$.
\end{theorem}
\begin{proof}
Observe that $v_1(p\vee{\sim}p,w)\!>\!0$ and $v_2(p\vee{\sim}p,w)\!<\!1$ for every $w\!\in\!\mathfrak{F}$. Since $\mathfrak{F}$ is finitely branching, $\inf\limits_{w'\in W}\{wSw'\rightarrow_{\mathsf{G}}v_1(p\vee{\sim}p,w')\}>0$ and $\sup\limits_{w'\in W}\{wSw'\wedge_{\mathsf{G}}v_2(p\vee{\sim}p,w')\}<1$ for $S\in\{R^+,R^-\}$. Thus, $v_1(\Box(p\vee{\sim}p),w)>0$ and $v_2(\Box(p\vee{\sim}p),w)<1$, whence, $v({\sim\sim}\Box(p\vee{\sim}p),w)=(1,0)$. Likewise, $v_1(\lozenge\neg(p\vee{\sim}p),w)<1$ and $v_2(\lozenge\neg(p\vee{\sim}p),w)>0$, whence $v(\mathbf{1}\coimplies\lozenge\neg(p\vee{\sim}p),w)=(1,0)$.

For the converse, we have two cases: (1) $|R^+(w)|\geq\aleph_0$ or (2) $|R^-(w)|\geq\aleph_0$ for some $w\in\mathfrak{F}$. In the first case, we let $X\subseteq R^+(w)$ be countable and define the value of $p$ as follows: $v(p,w'')=(1,0)$ for every $w''\notin X$ and $v(p,w_i)=\left(wR^+w'\cdot\frac{1}{i},0\right)$. It is clear that $\inf\limits_{w'\in W}\{wR^+w'\rightarrow_\mathsf{G}v_1(p\vee{\sim}p,w')\}=0$, whence $v_1({\sim\sim}\Box(p\vee{\sim}p))=0$ as required.

In the second case, $Y\subseteq R^-(w)$ be countable and define the value of $p$ as follows: $v(p,w'')=(1,0)$ for every $w''\notin Y$ and $v(p,w_i)=\left(wR^-w'\cdot\frac{1}{i},0\right)$. It is clear that $\inf\limits_{w'\in W}\{wR^-w'\rightarrow_\mathsf{G}v_2(\neg(p\vee{\sim}p),w')\}=0$, whence $v_2(\lozenge\neg(p\vee{\sim}p))=0$ and $v_2(\mathbf{1}\coimplies\lozenge\neg(p\vee{\sim}p))=1$ as required.
\end{proof}
\section{Constraint tableaux\label{sec:tableaux}}
Proof systems can be roughly divided into two kinds (cf.~\cite[\S2.4]{Haehnle2001HBPL} and the referenced literature for more details): the internal ones that produce proofs consisting of formulas only and the external ones whose proofs contain not only formulas but other entities as well: explicitly mentioned truth values, states, relations between formulas, etc. The standard examples are Hilbert and Gentzen calculi for the former and tableaux and labelled systems for the latter.

The external calculi usually formalise the semantic conditions straightforwardly which makes them suitable for establishing decidability and complexity results. Furthermore, their soundness and completeness are mostly proven by a routine check of said rules. Finally, external calculi usually allow for the explicit extraction of counter-models from failed proofs of invalid formulas.

In this paper, we present constraint tableaux for $\fbbirelKGsquare$. The original idea is due to~\cite{Haehnle1992,Haehnle1994,Haehnle1999} where constraint tableaux were used as a decision procedure for the propositional \L{}ukasiewicz logic $\Luk$. A similar calculus for the Rational Pawe\l{}ka logic was proposed in~\cite{diLascioGisolfi2005}. In~\cite{BilkovaFrittellaKozhemiachenko2021}, we presented constraint tableaux for $\Luk^2$ and $\Gsquare$ --- the paraconsistent expansions of $\Luk$ and $\mathsf{G}$. The constraint tableaux for finitely branching $\KbiG$ and $\KGsquare$ are provided in~\cite{BilkovaFrittellaKozhemiachenko2022IJCAR}. The present calculus which we call $\mathcal{T}\!\left(\fbbirelKGsquare\right)$ is an adaptation of $\mathcal{T}\!\left(\fbKGsquare\right)$~\cite{BilkovaFrittellaKozhemiachenko2022IJCAR}.
\begin{definition}\label{def:TfbbirelKGsquare}
We fix a set of state-labels $\mathsf{W}$ and let $\lesssim\in\!\{<,\leqslant\}$ and $\gtrsim\in\!\{>,\geqslant\}$. Let further $w\!\in\!\mathsf{W}$, $\mathbf{x}\!\in\!\{1,2\}$, $\phi\!\in\!\bimodalLsquare$, and $c\!\in\!\{0,1\}$. A~\emph{structure} is either $w\!:\!\mathbf{x}\!:\!\phi$, $c$, $w\mathsf{R}^+w'$, or $w\mathsf{R}^+w'$. We denote the set of structures with $\Str$. Structures of the form $w\!:\!\mathbf{x}\!:\!p$, $w\mathsf{R}^+w'$, and $w\mathsf{R}^-w'$ are called \emph{atomic} (denoted $\AStr$).

We define a \emph{constraint tableau} as a downward branching tree whose branches are sets containing constraints $\mathfrak{X}\lesssim\mathfrak{X'}$ ($\mathfrak{X},\mathfrak{X'}\in\Str$). Each branch can be extended by an application of a~rule\footnote{If $\mathfrak{X}\!<\!1,\mathfrak{X}\!<\!\mathfrak{X}'\!\in\!\mathcal{B}$ or $0\!<\!\mathfrak{X}',\mathfrak{X}\!<\!\mathfrak{X}'\!\in\!\mathcal{B}$, the rules are applied only to $\mathfrak{X}\!<\!\mathfrak{X}'$.} below (bars denote branching, $i,j\in\{1,2\}$, $i\neq j$, $w\mathsf{R}^+w'$ and $w\mathsf{R}^-w'$ occur on the branch, $w''$ is fresh on the branch).
\[\scriptsize{\begin{array}{cccc}
\neg_i\!\lesssim\!\dfrac{w\!:\!i\!:\!\neg\phi\!\lesssim\!\mathfrak{X}}{w\!:\!j\!:\!\phi\!\lesssim\!\mathfrak{X}}
&
\neg_i\!\gtrsim\!\dfrac{w\!:\!i\!:\!\neg\phi\!\gtrsim\!\mathfrak{X}}{w\!:\!j\!:\!\phi\!\gtrsim\!\mathfrak{X}}
&
\rightarrow_1\!\leqslant\!\dfrac{w\!:\!1\!:\!\phi\!\rightarrow\!\phi'\!\leqslant\!\mathfrak{X}}{\mathfrak{X}\!\geqslant\!{1}\left|\begin{matrix}\mathfrak{X}\!<\!{1}\\w\!:\!1\!:\!\phi'\!\leqslant\!\mathfrak{X}\\w\!:\!1\!:\!\phi\!>\!w\!:\!1\!:\!\phi'\end{matrix}\right.}
&
\rightarrow_2\!\geqslant\!\dfrac{w\!:\!2\!:\!\phi\rightarrow\phi'\!\geqslant\!\mathfrak{X}}{\mathfrak{X}\!\leqslant\!{0}\left|\begin{matrix}\mathfrak{X}\!>\!{0}\\w\!:\!2\!:\!\phi'\!\geqslant\!\mathfrak{X}\\w\!:\!2\!:\!\phi'\!>\!w\!:\!2\!:\!\phi\end{matrix}\right.}
\end{array}}\]

\[\scriptsize{\begin{array}{cccc}
\wedge_1\!\gtrsim\!\dfrac{w\!:\!1\!:\!\phi\!\wedge\!\phi'\!\gtrsim\!\mathfrak{X}}{\begin{matrix}w\!:\!1\!:\!\phi\!\gtrsim\!\mathfrak{X}\\w\!:\!1\!:\!\phi'\!\gtrsim\!\mathfrak{X}\end{matrix}}
&
\wedge_2\!\lesssim\!\dfrac{w\!:\!2\!:\!\phi\!\wedge\!\phi'\!\lesssim\!\mathfrak{X}}{\begin{matrix}w\!:\!2\!:\!\phi\!\lesssim\!\mathfrak{X}\\w\!:\!2\!:\!\phi'\!\lesssim\!\mathfrak{X}\end{matrix}}
&
\rightarrow_1\!<\!\dfrac{w\!:\!1\!:\!\phi\rightarrow\phi'\!<\!\mathfrak{X}}{\begin{matrix}w\!:\!1\!:\!\phi'\!<\!\mathfrak{X}\\w\!:\!1\!:\!\phi\!>\!w\!:\!1\!:\!\phi'\end{matrix}}
&
\rightarrow_2\!>\!\dfrac{w\!:\!2\!:\!\phi\rightarrow\phi'\!>\!\mathfrak{X}}{\begin{matrix}w\!:\!2\!:\!\phi'\!>\!\mathfrak{X}\\w\!:\!2\!:\!\phi'\!>\!w\!:\!2\!:\!\phi\end{matrix}}
\end{array}}\]

\[\scriptsize{\begin{array}{cc}
\wedge_1\!\lesssim\!\dfrac{w\!:\!1\!:\!\phi\wedge\phi'\!\lesssim\!\mathfrak{X}}{w\!:\!1\!:\!\phi\!\lesssim\!\mathfrak{X}\mid w\!:\!1\!:\!\phi'\!\lesssim\!\mathfrak{X}}
&\quad
\wedge_2\!\gtrsim\!\dfrac{w\!:\!2\!:\!\phi\wedge\phi'\!\gtrsim\!\mathfrak{X}}{w\!:\!2\!:\!\phi\!\gtrsim\!\mathfrak{X}\mid w\!:\!2\!:\!\phi'\!\gtrsim\!\mathfrak{X}}
\end{array}}\]

\[\scriptsize{\begin{array}{cc}
\rightarrow_1\!\gtrsim\!\dfrac{w\!:\!1\!:\!\phi\!\rightarrow\!\phi'\!\gtrsim\!\mathfrak{X}}{w\!:\!1\!:\!\phi\!\leqslant\!w\!:\!1\!:\!\phi'\mid w\!:\!1\!:\!\phi'\!\gtrsim\!\mathfrak{X}}&\rightarrow_2\!\lesssim\!\dfrac{w\!:\!2\!:\!\phi\rightarrow\phi'\!\lesssim\!\mathfrak{X}}{w\!:\!2\!:\!\phi'\!\leqslant\!w\!:\!2\!:\!\phi\mid w\!:\!2\!:\!\phi'\!\lesssim\!\mathfrak{X}}
\end{array}}\]

\[\scriptsize{\begin{array}{ccc}
\Box_1\!\!\gtrsim\!\dfrac{w\!:\!1\!:\!\Box\phi\!\gtrsim\!\mathfrak{X}}{w'\!:\!1\!:\!\phi\gtrsim\mathfrak{X}\mid w\mathsf{R}^+w'\!\leqslant\!w'\!:\!1\!:\!\phi}
&\quad
\Box_1\!\!\leqslant\!\dfrac{w\!:\!1\!:\!\Box\phi\!\leqslant\!\mathfrak{X}}{\mathfrak{X}\geqslant1\left|\begin{matrix}\mathfrak{X}\!<\!1\\w\mathsf{R}^+w''\!>\!w''\!:\!1\!:\!\phi\\w''\!:\!:\!1\!:\!\phi\leqslant\mathfrak{X}\end{matrix}\right.}
&\quad
\Box_1\!\!<\!\dfrac{w\!:\!1\!:\!\Box\phi\!<\!\mathfrak{X}}{\begin{matrix}w\mathsf{R}^+w''\!>\!w''\!:\!1\!:\!\phi\\w''\!:\!:\!1\!:\!\phi\!<\!\mathfrak{X}\end{matrix}}
\end{array}}\]

\[\scriptsize{\begin{array}{cccc}
\lozenge_1\!\!\gtrsim\!\dfrac{w\!:\!1\!:\!\lozenge\phi\!\gtrsim\!\mathfrak{X}}{\begin{matrix}w\mathsf{R}^+w''\!\gtrsim\!\mathfrak{X}\\w''\!:\!1\!:\!\phi\!\gtrsim\!\mathfrak{X}\end{matrix}}
&
\lozenge_1\!\!\lesssim\!\dfrac{w\!:\!1\!:\!\lozenge\phi\!\lesssim\!\mathfrak{X}}{w'\!:\!1\!:\!\phi\lesssim\mathfrak{X}\mid w\mathsf{R}^+w'\!\lesssim\!\mathfrak{X}}&
\Box_2\!\!\gtrsim\!\dfrac{w\!:\!2\!:\!\Box\phi\!\gtrsim\!\mathfrak{X}}{\begin{matrix}w\mathsf{R}^-w''\!\gtrsim\!\mathfrak{X}\\w''\!:\!2\!:\!\phi\!\gtrsim\!\mathfrak{X}\end{matrix}}
&
\Box_2\!\!\lesssim\!\dfrac{w\!:\!2\!:\!\lozenge\phi\!\lesssim\!\mathfrak{X}}{w'\!:\!2\!:\!\phi\lesssim\mathfrak{X}\mid w\mathsf{R}^-w'\!\lesssim\!\mathfrak{X}}
\end{array}}\]

\[\scriptsize{\begin{array}{ccc}
\lozenge_2\!\!\gtrsim\!\dfrac{w\!:\!2\!:\!\lozenge\phi\!\gtrsim\!\mathfrak{X}}{w'\!:\!2\!:\!\phi\gtrsim\mathfrak{X}\mid w\mathsf{R}^-w'\!\leqslant\!w'\!:\!1\!:\!\phi}
&\quad
\lozenge_2\!\!\leqslant\!\dfrac{w\!:\!2\!:\!\lozenge\phi\!\leqslant\!\mathfrak{X}}{\mathfrak{X}\geqslant1\left|\begin{matrix}\mathfrak{X}\!<\!1\\w\mathsf{R}^-w''\!>\!w''\!:\!2\!:\!\phi\\w''\!:\!:\!2\!:\!\phi\leqslant\mathfrak{X}\end{matrix}\right.}
&\quad
\lozenge_2\!\!<\!\dfrac{w\!:\!2\!:\!\lozenge\phi\!<\!\mathfrak{X}}{\begin{matrix}w\mathsf{R}^-w''\!>\!w''\!:\!2\!:\!\phi\\w''\!:\!:\!2\!:\!\phi\!<\!\mathfrak{X}\end{matrix}}
\end{array}}\]

A tableau's branch $\mathcal{B}$ is \emph{closed} iff one of the following conditions applies:
\begin{itemize}[noitemsep,topsep=2pt]
\item the transitive closure of $\mathcal{B}$ under $\lesssim$ contains $\mathfrak{X}<\mathfrak{X}$;
\item ${0}\geqslant{1}\in\mathcal{B}$, or $\mathfrak{X}>{1}\in\mathcal{B}$, or $\mathfrak{X}<{0}\in\mathcal{B}$.
\end{itemize}
A tableau is \emph{closed} iff all its branches are closed. We say that there is a \emph{tableau proof} of $\phi$ iff there are closed tableaux starting from $w\!:\!1\!:\!\phi<1$ and $w\!:\!2\!:\!\phi>0$.

An open branch $\mathcal{B}$ is \emph{complete} iff the following condition is met.
\begin{itemize}[noitemsep,topsep=2pt]
\item[$*$]\emph{If all premises of a rule occur on $\mathcal{B}$, then its one conclusion\footnote{Note that branching rules have \emph{two} conclusions.} occurs on~$\mathcal{B}$.}
\end{itemize}
\end{definition}

Before proceeding to the completeness proof, let us explain how $\mathcal{T}\!\left(\fbbirelKGsquare\right)$ works. First, we summarise the meanings of tableaux entries in the table below.
\begin{center}
\begin{tabular}{c|c}
\textbf{entry}&\textbf{interpretation}\\\hline
$w\!:1\!:\!\phi\leqslant w'\!:2\!:\!\phi'$&$v_1(\phi,w)\leq v_2(\phi',w')$\\
$w\!:\!2\!:\!\phi\leqslant c$&$v_2(\phi,w)\leq c$ with $c\in\{0,1\}$\\
$w\mathsf{R}^-w'\leqslant w'\!:2\!:\!\phi$&$wR^-w'\leq v_2(\phi,w')$
\end{tabular}
\end{center}
\begin{definition}[Branch realisation]\label{G2branchsatisfaction}
A model $\mathfrak{M}=\langle W,R^+,R^-,v_1,v_2\rangle$ with $W=\{w:w\text{ occurs on }\mathcal{B}\}$ \emph{realises a~branch $\mathcal{B}$} of a tableau iff there is a function $\real:\Str\rightarrow[0,1]$ s.t.\ for every $\mathfrak{X},\mathfrak{Y},\mathfrak{Y}',\mathfrak{Z},\mathfrak{Z}'\in\Str$ with $\mathfrak{X}=w:\mathbf{x}:\phi$, $\mathfrak{Y}=w_i\mathsf{R}^+w_j$, and $\mathfrak{Y}'=w'_i\mathsf{R}^-w'_j$ the following holds ($\mathbf{x}\in\{1,2\}$, ${c}\in\{0,1\}$).
\begin{itemize}[noitemsep,topsep=2pt]
\item If $\mathfrak{Z}\lesssim\mathfrak{Z}'\in\mathcal{B}$, then $\real(\mathfrak{Z})\lesssim\real(\mathfrak{Z}')$.
\item $\real(\mathfrak{X})=v_\mathbf{x}(\phi,w)$, $\real(c)=c$, $\real(\mathfrak{Y})=w_iR^+w_j$, $\real(\mathfrak{Y}')=w'_iR^-w'_j$
\end{itemize}
\end{definition}
Let us now provide an example of a failed proof with a complete open branch (marked with $\frownie$ below) and construct a model realising it.
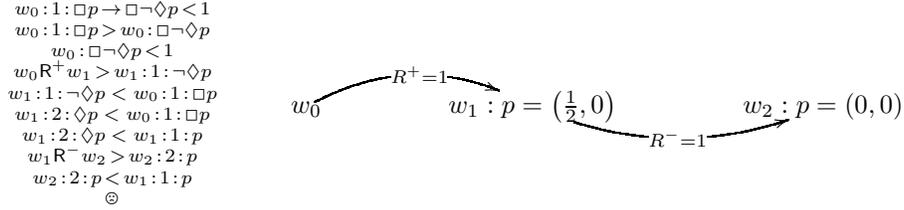
\begin{center}
\begin{minipage}{.25\linewidth}
\scriptsize{
\begin{forest}
smullyan tableaux
[w_0\!:\!1\!:\!\Box p\!\rightarrow\!\Box\neg\lozenge p\!<\!1
[w_0\!:\!1\!:\!\Box p\!>\!w_0\!:\!\Box\neg\lozenge p
[w_0\!:\!\Box\neg\lozenge p\!<\!1
[w_0\mathsf{R}^+w_1\!>\!w_1\!:\!1\!:\!\neg\lozenge p
[w_1\!:\!1\!:\!\neg\lozenge p<w_0\!:\!1\!:\!\Box p
[w_1\!:\!2\!:\!\lozenge p<w_0\!:\!1\!:\!\Box p
[w_1\!:\!2\!:\!\lozenge p<w_1\!:\!1\!:\!p
[w_1\mathsf{R}^-w_2\!>\!w_2\!:\!2\!:\! p
[w_2\!:\!2\!:\! p\!<\!w_1\!:\!1\!:\!p[\frownie]]
]]]]]]]]
\end{forest}}
\end{minipage}
\hfill
\begin{minipage}{.70\linewidth}
\[\xymatrix{w_0\ar@/^1pc/[rr]|{R^+=1}&&w_1:p=\left(\frac{1}{2},0\right)\ar@/_1pc/[rr]|{R^-=1}&&w_2:p=(0,0)}\]
\end{minipage}
\end{center}
The proof goes as follows: first, we apply all the possible propositional rules, then the modal rules that introduce new states, and then the modal rules using the newly introduced states. We repeat the process until we decompose all structures into atoms.

We then extract a model from the complete open branch marked with $\frownie$ s.t.\ $v_1(\Box p\!\rightarrow\!\Box\neg\lozenge p,w_0)<1$. We use $w$'s on the branch as the carrier and assign the values of variables and relations so that they correspond to $\lesssim$.
\begin{theorem}[$\mathcal{T}\!\left(\fbbirelKGsquare\right)$ completeness]\label{theorem:T+-KG2completeness}
$\phi$ is $\birelKGsquare$ valid iff there is a tableau proof of $\phi$.
\end{theorem}
\begin{proof}
We adapt the technique from~\cite[Theorem~3]{BilkovaFrittellaKozhemiachenko2022IJCAR} and consider only the most important cases.

For soundness, we prove that if the premise of the rule is realised, then so is at least one of its conclusions. Note that since we work with finitely branching frames, infima and suprema from Definition~\ref{def:semantics} become maxima and minima. Since propositional rules are exactly the same as in $\mathcal{T}\left(\fbKGsquare\right)$~\cite{BilkovaFrittellaKozhemiachenko2022IJCAR}, we consider only the most interesting cases of modal rules. We tackle $\Box_1\!\!\gtrsim$ (cf.\ Definition~\ref{def:TfbbirelKGsquare}) and show that if $\mathfrak{M}=\langle W,R^+,R^-,v_1,v_2\rangle$ realises the premise of the rule, it also realises one of its conclusions.

Assume w.l.o.g.\ that $\mathfrak{X}=w''\!:\!2\!:\!\psi$, and let $\mathfrak{M}$ realise $w\!:\!1\!:\!\Box\phi\geqslant w''\!:\!2\!:\!\psi$. Now, since $R^+$ is finitely branching, we have $\min\limits_{w'\in W}\{w\mathsf{R}^+w'\rightarrow_\mathsf{G}v_1(\phi,w')\}\geq v_2(\psi,w)$, whence at each $w'\in W$ s.t.\ $wR^+w'>0$\footnote{Recall that if $u\mathsf{R}^+u'\notin\mathcal{B}$, we set $u\mathsf{R}^+u'=0$.}, either $v_1(\phi,w')\geq v_2(\psi,w'')$ or $w\mathsf{R}^+w'\geq v_2(\psi,w'')$. Thus, at least one conclusion of the rule is satisfied.

Other rules can be dealt with similarly. Since closed branches are not realisable, the result follows.

To prove completeness, we show that a realising model can be built for every complete open branch $\mathcal{B}$. First, we set $W=\{w:w\text{ occurs in }\mathcal{B}\}$. Denote the set of atomic structures appearing on $\mathcal{B}$ with $\AStr(\mathcal{B})$ and let $\mathcal{B}^+$ be the transitive closure of $\mathcal{B}$ under $\lesssim$. Now, we assign values. For $i\!\in\!\{1,2\}$, if $w\!:\!i\!:\!p\!\geqslant\!1\!\in\!\mathcal{B}$, we set $v_i(p,w)\!=\!1$. If $w\!:\!i\!:\!p\!\leqslant\!0\!\in\!\mathcal{B}$, we set $v_i(p,w)\!=\!0$. If $w\mathsf{S}w'\!<\!\mathfrak{X}\!\notin\!\mathcal{B}^+$, we set $w\mathsf{S}w'\!=\!1$. If $w\!:\!i\!:\!p$ or $w\mathsf{S}w'$ does not occur on $\mathcal{B}$, we set $v_i(p,w)\!=\!0$ and $w\mathsf{S}w'\!=\!0$.

For each $\str\in\AStr$, we now set
\[[\str]\!=\!\left\{\str'\left| \; \begin{matrix}\str\leqslant\str'\in\mathcal{B}^+\text{ and }\str<\str
\notin\mathcal{B}^+\\
\text{or}\\
\str\geqslant\str'\in\mathcal{B}^+\text{ and }\str>\str'\notin\mathcal{B}^+
\end{matrix}\right.\right\}\]
Denote the number of $[\str]$'s with $\#^\mathsf{str}$. Since the only possible loop in $\mathcal{B}^+$ is $\str\leqslant\str'\leqslant\ldots\leqslant\str$ where all elements belong to $[\str]$, it is clear that $\#^\mathsf{str}\leq2\cdot|\AStr(\mathcal{B})|\cdot|W|$. Put $[\str]\prec[\str']$ iff there are $\str_i\in[\str]$ and $\str_j\in[\str']$ s.t.\ $\str_i<\str_j\in\mathcal{B}^+$. We now set the valuation of these structures as follows:
\begin{align*}
\str=\dfrac{|\{[\str']:[\str']\prec[\str]\}|}{\#^\mathsf{str}}
\end{align*}
It is clear that constraints containing only atomic structures and constants are now satisfied. To show that all other constraints are satisfied, we prove that if at least one conclusion of the rule is satisfied, then so is the premise. The proof is done by considering the cases of rules. We consider only the case of $\Box_1\!\!\gtrsim$ and assume w.l.o.g.\ that $\mathfrak{X}=w''\!:\!2\!:\!\psi$.

For $\Box_1\!\!\gtrsim$, assume that for every $w'$ s.t.\ $w\mathsf{R}^+w'$ is on the branch, either $w'\!:\!1\!:\!\phi\geqslant w''\!:\!2\!:\!\psi$ or $w\mathsf{R}^+w'\leqslant w'\!:\!1\!:\!\phi$ is realisable. Thus, by the inductive hypothesis, for every $w'\in R^+(w)$, it holds that $v_1(\phi,w')\geq v_2(\psi,w'')$ or $wR^+w'\leq v_1(\phi,w')$. Hence, $v_1(\Box\phi,w)\geq v_2(\psi,w'')$ and $w\!:\!1\!:\!\Box\phi\geqslant w''\!:\!2\!:\!\psi$ is realised.
\end{proof}

We can now use $\mathcal{T}\!\left(\fbbirelKGsquare\right)$ to obtain the expected finite model property and upper bound on the size of satisfying (or falsifying\footnote{Satisfiability and falsifiability (non-validity) are reducible to each other: $\phi$ is satisfiable iff ${\sim\sim}(\phi\coimplies\mathbf{0})$ is falsifiable; $\phi$ is falsifiable iff ${\sim\sim}(\mathbf{1}\coimplies\phi)$ is satisfiable.}) models.
\begin{corollary}\label{cor:FMP}
Let $\phi\in\bimodalLsquare$ be \emph{not $\fbbirelKGsquare$ valid}, and let $k$ be the number of modalities in it. Then there is a model $\mathfrak{M}$ of the size $\leq k^{k+1}$ and depth $\leq k$ and $w\in\mathfrak{M}$ s.t.\ $v_1(\phi,w)\neq1$ or $v_2(\phi,w)\neq0$.
\end{corollary}
\begin{proof}
By theorem~\ref{theorem:T+-KG2completeness}, if $\phi$ is \emph{not $\fbbirelKGsquare$ valid}, we can build a~falsifying model using tableaux. It is also clear from the rules in Definition~\ref{def:TfbbirelKGsquare} that the depth of the constructed model is bounded from above by the maximal number of nested modalities in $\phi$. The width of the model is bounded by the maximal number of modalities on the same level of nesting.
\end{proof}

Now, using the upper bound on the size of the model, we can reduce the satisfiability in $\fbbirelKGsquare$ to the satisfiability on the models where the values of subformulas and relations are over some finite bi-G\"{o}del algebra. This allows us to avoid comparisons of formulas in different states, whence, we can build the satisfying model branch by branch. We adapt the algorithm from~\cite{BilkovaFrittellaKozhemiachenko2022IJCAR}.
\begin{theorem}\label{theorem:fbbirelKG2PSPACE}
$\fbbirelKGsquare$ validity and satisfiability are $\mathsf{PSPACE}$ complete.
\end{theorem}
\begin{proof}
For the membership, observe from the proof of Theorem~\ref{theorem:T+-KG2completeness} that $\phi$ is satisfiable (falsifiable) on $\mathfrak{M}=\langle W,R^+,R^-,v_1,v_2\rangle$ iff all variables, $w\mathsf{R}^+w'$'s, and $w\mathsf{R}^-w'$'s have values from $\mathsf{V}=\left\{0,\frac{1}{\#^\str},\ldots,\frac{\#^\str-1}{\#^\str},1\right\}$ under which $\phi$ is satisfied (falsified).

Since $\#^\str$ is bounded from above, we can now replace constraints with labelled formulas and relational structures of the form $w\!:\!i\!:\!\phi\!=\!\mathsf{v}$ or $w\mathsf{S}w'\!=\!\mathsf{v}$ ($\mathsf{v}\in\mathsf{V}$) avoiding comparisons of values of formulas in different states. We close the branch if it contains $w\!:\!i\!:\!\psi\!=\!\mathsf{v}$ and $w\!:\!i\!:\!\psi\!=\!\mathsf{v}'$ for $\mathsf{v}\!\neq\!\mathsf{v}'$.

Now we replace the $\mathcal{T}\!\left(\fbbirelKGsquare\right)$ rules with ones that work with labelled structures. Below, we give as an example the rules\footnote{For a value $\mathsf{v}>0$ of $\lozenge\phi$ at $w$, we add a new state that witnesses $\mathsf{v}$, and for a~state on the branch, we guess a~value smaller than $\mathsf{v}$. Other modal rules can be rewritten similarly.} that replace $\lozenge_1\!\!\lesssim$.
\begin{center}
\scriptsize{\begin{align*}
\dfrac{w\!:\!1\!:\!\lozenge\phi\!=\!\frac{r}{\#^\str}}{\left.\begin{matrix}w\mathsf{R}^+w'\!=\!1\\w\!:\!1\!:\!\phi\!=\!\frac{r}{\#^\str}\end{matrix}\right|\left.\begin{matrix}w\mathsf{R}^+w'\!=\!\frac{r}{\#^\str}\\w\!:\!1\!:\!\phi\!=\!1\end{matrix}\right|\ldots\left|\begin{matrix}w\mathsf{R}^+w'\!=\!\frac{r}{\#^\str}\\w\!:\!1\!:\!\phi\!=\!\frac{r}{\#^\str}\end{matrix}\right.}
&&
\dfrac{w\!:\!i\!:\!\lozenge\phi\!=\!\frac{r}{\#^\str};(w\mathsf{R}^+w'\text{ occurs on the branch})}{w'\!:\!i\!:\!\phi\!=\!\frac{r-1}{\#^\str}\mid w\mathsf{R}^+w'\!=\!\frac{r-1}{\#^\str}\mid\ldots\mid w'\!:\!i\!:\!\phi\!=\!0}\end{align*}}
\end{center}
Note that the rules of such form prevent us from comparing values of formulas in different states.

We can now build a satisfying model for $\phi$ using polynomial space. We begin with $w_0\!:\!1\!:\phi\!=\!1$ (the algorithm for $w_0\!:\!1\!:\phi\!=\!0$ is the same) and start applying propositional rules (first, those that do not require branching). If we implement a branching rule, we pick one branch and work only with it: either until the branch is closed, in which case we pick another one; until no more rules are applicable (then, the model is constructed); or until we need to apply a modal rule to proceed. At this stage, we need to store only the subformulas of $\phi$ with labels denoting their value at~$w_0$.

Now we guess a~modal formula (say, $w_0\!:\!1\!:\!\lozenge\chi\!=\!\frac{1}{\#^\str}$) whose decomposition requires an introduction of a~new state ($w_1$) and apply this rule. Then we apply all modal rules whose implementation requires that $w_0\mathsf{R}^+w_1$ occur on the branch (again, if those require branching, we guess only one branch) and start from the beginning with the propositional rules. If we reach a contradiction, the branch is closed. Again, the only new entries to store are subformulas of $\phi$ (now, with fewer modalities), their values at $w_1$, and a~relational term $w_0\mathsf{R}^+w_1$ with its value. Since the depth of the model is $O(|\phi|)$ and since we work with modal formulas one by one, we need to store subformulas of $\phi$ with their values $O(|\phi|)$ times, so, we need only $O(|\phi|^2)$ space.

Finally, if no rule is applicable and there is no contradiction, we mark $w_0\!:\!2\!:\!\lozenge\chi\!=\!\frac{1}{\#^\str}$ as ‘safe’. Now we \emph{delete all entries of the tableau below it} and pick another unmarked modal formula that requires an introduction of a new state. Dealing with them one by one allows us to construct the model branch by branch. But since the length of each branch of the model is bounded by $O(|\phi|)$ and since we delete \emph{branches of the model} once they are shown to contain no contradictions, we need only polynomial space.

To establish $\pspace$ hardness, we provide a reduction from $\mathbf{K}$ validity. Namely, for every $\phi$ over $\{\mathbf{0},\wedge,\vee,\rightarrow,\Box,\lozenge\}$, we construct two formulas:
\begin{enumerate}[noitemsep,topsep=2pt]
\item $\phi^\triangledown$ s.t.\ $\mathbf{K}\models\phi$ iff for every frame $\mathfrak{F}$, valuation $v_1$ and $w\in\mathfrak{F}$, $v_1(\phi^\triangledown,w)=1$;
\item $\phi^\triangle$ s.t.\ $\mathbf{K}\models\phi$ iff for every frame $\mathfrak{F}$, valuation $v_2$ and $w\!\in\!\mathfrak{F}$, $v_2(\mathbf{1}\!\coimplies\!\phi^\triangle,w)\!=\!0$.
\end{enumerate}
To construct $\phi^\triangledown$, we borrow the idea from~\cite[Theorem~21]{CaicedoMetcalfeRodriguezRogger2017} and put ${\sim\sim}$ in front of every subformula of $\phi$. Since semantics of $\KG$ coincides with $v_1$ conditions in Definition~\ref{def:semantics}, the reduction holds.

For $\phi^\triangle$, we use the following inductive definition
\begin{align*}
p^\triangle&=\triangle p\\
(\chi\circ\psi)^\triangle&=\chi^\triangle\bullet\psi^\triangle\tag{$\circ,\bullet\in\{\wedge,\vee\}$, $\circ\neq\bullet$}\\
(\chi\rightarrow\psi)^\triangle&=\psi^\triangle\coimplies\chi^\triangle\\
(\heartsuit\chi)^\triangle&=\spadesuit(\chi^\triangle)\tag{$\heartsuit\neq\spadesuit$, $\heartsuit,\spadesuit\in\{\Box,\lozenge\}$}
\end{align*}
One can check by induction that for every \emph{crisp} finitely branching $\mathfrak{F}$ and every \emph{classical} valuation $\mathbf{v}$ thereon, it holds that $\mathfrak{F},\mathbf{v},w\vDash\phi$ iff $v_2(\mathbf{1}\coimplies\phi^\triangle,w)=0$ and $\mathfrak{F},\mathbf{v},w\nvDash\phi$ iff $v_2(\mathbf{1}\coimplies\phi^\triangle,w)=1$ provided that $v_2=\mathbf{v}$.

For the converse, let $\mathfrak{M}=\langle W,R^+,R^-,v_1,v_2\rangle$ be a $\fbbirelKGsquare$ model. Let $\mathfrak{M}^!=\langle W,R^!,v^!\rangle$ be s.t.\ $wR^!w'$ iff $wR^-w'=1$ and $w\in v^!(p)$ iff $v_2(p,w)=1$. Again, it is easy to verify that for every $\mathfrak{M}$, $v_2(\phi^\triangle,w)=1$ iff $\mathfrak{M}^!,w\vDash\phi$.

It now follows that
\begin{enumerate}[noitemsep,topsep=2pt]
\item if $\phi$ is \emph{not $\mathbf{K}$-valid}, then there are a $\mathfrak{F}$, $w\!\in\!\mathfrak{F}$, and $v_2$ s.t.\ $v_2(\mathbf{1}\coimplies\phi^\triangle,w)\!\neq\!0$;
\item if there are $\mathfrak{F}$, $w\!\in\!\mathfrak{F}$, and $v_2$ s.t.\ $v_2(\mathbf{1}\!\coimplies\!\phi^\triangle,w)\!\neq\!0$, then $\phi$ is \emph{not $\mathbf{K}$-valid}.
\end{enumerate}
\end{proof}
\section{Conclusion\label{sec:conclusion}}
In this paper, we developed paraconsistent G\"{o}del modal logics interpreted on bi-relational fuzzy and crisp frames. Namely, we proved (Theorems~\ref{theorem:noextension} and~\ref{theorem:crispextension}) that crisp $\birelKGsquare$ extends crisp bi-G\"{o}del modal logic $\KbiG$ while fuzzy $\birelKGsquare$ does not extend fuzzy $\KbiG$. Moreover, we proved that bi-relational counterparts of all $\KbiG$ definable classes of crisp frames are also $\birelKGsquare$ definable (Corollary~\ref{cor:definabilitypreservation}). We also proved that finitely branching crisp and fuzzy frames are definable. For $\birelKGsquare$ over finite frames ($\fbbirelKGsquare$), we constructed a sound and complete constraint tableaux calculus and proved that validity and satisfiability of $\fbbirelKGsquare$ are $\pspace$ complete.

Several questions remain open. First, we wish to axiomatise crisp and fuzzy $\birelKGsquare$. Second, while it is reasonable to conjecture that $\birelKGsquare$ over all frames is also $\pspace$-complete, it is not straightforward to adapt the technique from~\cite{CaicedoMetcalfeRodriguezRogger2017} since $\birelKGsquare$ does not extend $\KbiG$ nor has negation normal forms.

It also seems that ‘traditional’ definitions of classes of fuzzy frames ($\Box p\!\rightarrow\!p$ for reflexive frames, $\Box p\!\rightarrow\!\lozenge p$ for frames s.t.\ $R^+(w),R^-(w)\!\neq\!\varnothing$ for every $w$, etc.) define their $\pm$-counterparts in $\birelKGsquare$. Of course, it is not true in general: $\lozenge\phi$ also defines frames with $R(w)\!\neq\!\varnothing$ for every $w$ whenever $\phi$ is $\KbiG$ valid but $\birelKGsquare$ does not extend $\KbiG$, whence not every $\lozenge\phi$ can be used as such a~definition. Another question is whether it is always possible, given a~$\neg$-free $\phi$ and mono-relational $\mathfrak{F}$ s.t.\ $\mathfrak{F}\!\models_{\KbiG}\!\phi$, to produce a~bi-relational $\mathfrak{F}^\pm$ s.t.~$\mathfrak{F}^\pm\!\models\!\phi$.
\bibliographystyle{splncs04}
\bibliography{reference}
\end{document}